\documentclass[11pt]{article}
\usepackage{hyperref}

\usepackage{geometry}
\usepackage{times,authblk}
\usepackage{dutchcal}
\usepackage[UKenglish]{babel}
\usepackage{amsmath, amsfonts, amssymb, amsthm, bm, stmaryrd, enumitem, mathtools}
\usepackage{graphicx, tikz-cd}
\usepackage[font=small,labelfont=bf]{caption}
\renewcommand{\epsilon}{\varepsilon}
\renewcommand{\phi}{\varphi}
\renewcommand{\rho}{\varrho}

\newtheorem{Def}{Definition}[section]
\newenvironment{definition}{\begin{Def} \rm}{\end{Def}}
\newtheorem{lemma}[Def]{Lemma}
\newtheorem{proposition}[Def]{Proposition}
\newtheorem{corollary}[Def]{Corollary}
\newtheorem{theorem}[Def]{Theorem}
\newtheorem{example}[Def]{Example}

\newcommand{\komma}{,\hspace{0.3em}}
\newcommand{\id}{\text{id}}
\renewcommand{\leq}{\leqslant}
\renewcommand{\geq}{\geqslant}
\renewcommand{\emptyset}{\varnothing}
\newenvironment{smm}{\scriptsize \begin{pmatrix}}{\end{pmatrix}}

\newcommand{\Naturals}{{\mathbb N}}

\newcommand{\Rationals}{{\mathbb Q}}
\newcommand{\Reals}{{\mathbb R}}

\newcommand{\Quaternions}{{\mathbb H}}

\newcommand{\notperp}{\mathbin{\not\perp}}
\renewcommand{\c}{^\perp}
\newcommand{\cc}{^{\perp\perp}}
\newcommand{\herm}[2]{\left( #1 , #2 \right)}
\newcommand{\hermalt}[2]{\left( #1 , #2 \right)'}

\newcommand{\lin}[1]{\langle#1\rangle}
\newcommand{\linbig}[1]{\big\langle#1\big\rangle}
\newcommand{\linBig}[1]{\Big\langle#1\Big\rangle}
\newcommand{\withoutzero}{^{\raisebox{0.2ex}{\scalebox{0.4}{$\bullet$}}}}

\newcommand{\U}{\mathbf U}
\renewcommand{\O}{\mathbf O}
\DeclareMathOperator{\Aut}{Aut}
\newcommand{\N}{\mathbf N}
\newcommand{\SO}{\mathbf{SO}}
\newcommand{\downset}{{\downarrow}}

\begin{document}

\title{Characterisation of quadratic spaces over the Hilbert field \\ by means of the orthogonality relation}

\author{Miroslav Korbel\' a\v r}

\affil{\footnotesize
Department of Mathematics, Faculty of Electrical Engineering,  \authorcr
Czech Technical University in Prague, Technick\' a 2, 166 27 Praha, Czech Republic \authorcr
{\tt korbemir@fel.cvut.cz}}

\author{Jan Paseka}

\affil{\footnotesize
Department of Mathematics and Statistics, Masaryk~University, \authorcr
Kotl\'a\v rsk\'a 2, 611\,37 Brno, Czech Republic \authorcr
{\tt paseka@math.muni.cz}}

\author{Thomas Vetterlein}

\affil{\footnotesize
Institute for Mathematical Methods in Medicine and Data Based Modeling, \authorcr
Johannes Kepler University, Altenberger Stra\ss{}e 69, 4040 Linz, Austria \authorcr
{\tt Thomas.Vetterlein@jku.at}}

\date{\today}

\maketitle

\begin{abstract}\parindent0pt\parskip1ex

\noindent\vspace{-5ex}

An orthoset is a set equipped with a symmetric, irreflexive binary relation. With any (aniso\-tropic) Hermitian space $H$, we may associate the orthoset $(P(H),\perp)$, consisting of the set of one-dimensional subspaces of $H$ and the usual orthogonality relation. $(P(H),\perp)$ determines $H$ essentially uniquely.

We characterise in this paper certain kinds of Hermitian spaces by imposing transitivity and minimality conditions on their associated orthosets. By gradually considering stricter conditions, we restrict the discussion to a narrower and narrower class of Hermitian spaces. Ultimately, our interest lies in quadratic spaces over countable subfields of $\Reals$.

A line of an orthoset is the orthoclosure of two distinct elements. For an orthoset to be line-symmetric means roughly that its automorphism group acts transitively both on the collection of all lines as well as on each single line. Line-symmetric orthosets turn out to be in correspondence with transitive Hermitian spaces. Furthermore, quadratic orthosets are defined similarly, but are required to possess, for each line $\ell$, a group of automorphisms acting on $\ell$ transitively and commutatively. We show the correspondence of quadratic orthosets with transitive quadratic spaces over ordered fields. We finally specify those quadratic orthosets that are, in a natural sense, minimal: for a finite $n \geq 4$, the orthoset $(P(R^n),\perp)$, where $R$ is the Hilbert field, has the property of being embeddable into any other quadratic orthoset of rank~$n$.



\mbox{}\vspace{-2ex}

\end{abstract}

\section{Introduction}
\label{sec:Introduction}

The algebraic characterisation of Hilbert spaces is an issue that has been discussed in numerous papers from different viewpoints. A major motivation has been to gain a better understanding of the role of these structures for modelling quantum phenomena \cite{EGL1,EGL2}. Particular interest has focused on the lattice-theoretic approach, which is based on the idea of describing a Hilbert space by means of the ortholattice of its closed subspaces \cite{Kal}. The program was successful to some extent; there is a way of specifying the infinite-dimensional complex Hilbert space in the first-order language of ortholattices, together with the requirement that the lattice has infinite length \cite{Wil}. A closely related idea is not to consider order-theoretic aspects but to deal with {\it orthosets}, also known as {\it orthogonality spaces}, that is, with structures based on a single binary relation assumed to be symmetric and irreflexive \cite{Dac,Wlc}. Following Foulis and his collaborators, the intention is in this case to reduce the structure of the Hilbert space to its orthogonality relation. In fact, a Hilbert space can be reconstructed from its associated orthoset and possibilities to describe this type of orthosets can be found, e.g., in \cite{Vet1,Vet2,Rum}.

However, a really satisfactory solution of this issue is for certain well-known reasons difficult. The situation can be regarded as convenient for a more general class of inner-product spaces, namely, the (anisotropic) Hermitian spaces. A lattice-theoretic description of the latter has been known for a long time \cite{MaMa} and we moreover recently established that a simple combinatorial condition characterises the corresponding orthosets \cite{Vet3,PaVe1}. However, to characterise among the division rings the classical fields remains an uncompleted project. In the lattice-theoretic framework, orthomodularity and an infinite length together with a number of technical conditions do lead to success, as we have mentioned above \cite{Wil}, but the hypotheses are hard to motivate. In the framework of orthosets of infinite rank, orthomodularity is not easily expressible and hence the most powerful tool, Sol\` er's Theorem \cite{Sol}, seems difficult to apply. Apart from that, it would be desirable to find an approach that is not restricted to the case of infinite dimensions.

Our aim is to describe the real Hilbert spaces, in particular the finite-dimensional ones, by means of orthosets. Having in mind the role of this structure in quantum physics, our focus is on symmetries: we actually regard orthosets as the underlying sets of group actions. In our previous work \cite{Vet4}, we showed that the existence of certain symmetries implies a situation that comes quite close to the prototype. The conditions imposed on the orthosets were shown to lead to quadratic spaces over (commutative) ordered fields \cite[Theorem~7]{Vet4}. The class of fields that can occur, however, is still rather wide. In a next step, it would be desirable to find reasonable conditions ensuring that the field is Archimedean. For this concern, however, the given framework seems not really suitable. Some options can be found in the preceding papers: in \cite{Vet1}, we proposed the non-existence of certain quotients and in \cite{Vet4}, we considered the quasiprimitivity of the group generated by the simple rotations.

In the present work, we choose a different approach. First of all, we show that the inner-product spaces in which we are interested can be characterised solely on the basis of transitivity conditions. Our definition of {\it line-symmetry} is modelled on concepts known from incidence geometry; see, e.g., \cite{BDD}. A line in an orthoset $(X,\perp)$ is the orthoclosure of a two-element set and $(X,\perp)$ is said to be line-symmetric if it is line-transitive and line-permutable. The former condition means that the automorphism group of $(X,\perp)$ acts transitively on the lines and the latter condition means that, for any line $\ell$, the group of automorphisms keeping the orthocomplement of $\ell$ elementwise fixed acts transitively on $\ell$. In addition, we call an orthoset Boolean if any two distinct elements are orthogonal. We show that the orthosets arising from transitive Hermitian spaces of dimension $\geq 4$ are exactly the non-Boolean line-symmetric ones. Second, we sharpen the notion of line-symmetry: rather than just requiring a transitive action on any given line, we require the transitive action to be abelian. We call orthosets of this type {\it quadratic} because it turns out that they correspond to quadratic spaces. That is, non-Boolean quadratic orthosets of rank $\geq 4$ are associated with linear spaces over a (commutative) field equipped with a symmetric bilinear form.

Finally, we are interested in the ``minimal'' members of the class of quadratic orthosets of a given rank. That is, we raise the question whether, for some finite $n$, there is a quadratic orthoset of rank $n$ embedding into any other quadratic orthoset of this rank. This turns out to be indeed the case. Let $R$ be the Hilbert field, that is, the minimal Pythagorean subfield of $\Reals$. For a finite $n \geq 4$, the orthosets $(P(R^n),\perp)$ have the desired minimality property. We additionally answer negatively the question whether also the other classes of orthosets that we consider here have, in the indicated sense, minimal members.

The paper is structured as follows. In the subsequent Section~\ref{sec:Preliminaries}, we compile facts and definitions around Hermitian spaces and their corresponding orthosets, the linear orthosets. We also indicate in which way the respective automorphism groups are mutually related. In Section~\ref{sec:Transitivity}, we deal with line-symmetric orthosets and their correspondence with transitive Hermitian spaces. In Section~\ref{sec:quadratic-orthosets}, we introduce the class of orthosets in which we are primarily interested, the quadratic orthosets, and we establish their correspondence with transitive quadratic spaces. The concluding Section~\ref{sec:Minimal-orthosets} is devoted to ``minimal'' orthosets of a given type and rank. Based on a minimality condition we get a characterisation of quadratic spaces over the Hilbert field.

\section{Hermitian spaces and their associated orthosets}
\label{sec:Preliminaries}

We begin by recalling basic definitions on the considered type of inner-product spaces. We furthermore summarise basic facts on orthosets that naturally arise from linear spaces endowed with a symmetric and anisotropic inner product.

By a {\it $\star$-sfield}, we mean a sfield (i.e., a division ring) $F$ that is equipped with an involutive antiautomorphism of $F$. A linear space $H$ over a $\star$-sfield $F$ that is equipped with a Hermitian form is called a {\it Hermitian space}. Here, a {\it Hermitian form} is meant to be a map $\herm{\cdot}{\cdot} \colon H \times H \to F$ such that, for any $u, v, w \in H$ and $\alpha, \beta \in F$, we have
\begin{align*}
& \herm{\alpha u + \beta v}{w} \;=\; \alpha \, \herm{u}{w} + \beta \, \herm{v}{w}, \\
& \herm{w}{\alpha u + \beta v} \;=\;
                               \herm{w}{u} \, \alpha^\star + \herm{w}{v} \, \beta^\star, \\
& \herm{u}{v} \;=\; \herm{v}{u}^\star, \\
& \herm{u}{u} = 0 \text{ implies } u = 0,
\end{align*}
where $^\star$ denotes the involution on $F$. We note that the last condition, the anisotropy of the form, is not usually assumed. We refer to $H$ as a {\it quadratic space} in the case when $F$ is commutative and the involution $^\star$ on $F$ is the identity.

Let $H$ be a Hermitian space. Two vectors $u, v \in H$ such that $\herm u v = 0$ are called {\it orthogonal} and we write $u \perp v$ in this case. By the {\it dimension} of a Hermitian space we mean its Hilbert dimension, which depends solely on the orthogonality relation: the supremum of the cardinalities of sets consisting of mutually orthogonal non-zero vectors in $H$. We note that the finite-dimensional spaces are exactly those of finite algebraic dimension, in which case both values coincide. We tacitly assume in this paper Hermitian spaces to have an at most countable dimension.

We put $H\withoutzero = H \setminus \{0\}$. The subspace spanned by vectors $u_1, \ldots, u_k \in H\withoutzero$ is denoted by $\lin{u_1, \ldots, u_k}$. The {\it orthocomplement} of a set $M \subseteq H$ is $M\c = \{ u \in H \colon u \perp v \text{ for any } v \in M \}$ and subspaces of this form are called {\it orthoclosed}. We moreover recall that a subspace $M$ of $H$ is called {\it splitting} if $H = M + M\c$, that is, if $H$ is the direct sum of $M$ and $M\c$. Every finite-dimensional subspace is splitting and every splitting subspace is orthoclosed.

A Hermitian space $H$ can, in a sense that we shall make precise now, be reduced to a seemingly very simple type of structure. Let us consider
\[ P(H) \;=\; \{ \lin x \colon x \in H\withoutzero \}, \]
that is, the collection of one-dimensional subspaces of $H$. We intend to characterise $H$ by means of the inner structure of $P(H)$.

Recall that an (irreducible) {\it projective space} is a set $P$ equipped with an operation $\star$ sending each pair $e, f \in P$ to a set $e \star f \subseteq P$, such that, for any $e, f, g, h \in P$, the following conditions hold: (P1) $e, f \in e \star f$, and $e \star f$ contains an element distinct from $e$ and $f$ if and only if $e \neq f$; (P2) if $g, h \in e \star f$ and $g \neq h$, then $g \star h = e \star f$; (P3) for $e, f, g, h$ pairwise distinct, we have that ${e \star f} \,\cap\, {g \star h} \neq \emptyset$ implies ${e \star g} \,\cap\, {f \star h} \neq \emptyset$. The {\it rank} of a projective space $P$ is the minimal cardinality of a set $B \subseteq P$ whose closure under $\star$ is all of $P$. For a comprehensive account of projective spaces, we refer the reader to \cite{FaFr}.

$P(H)$ is a projective space in a natural way: for $u, v \in H\withoutzero$ we define
\begin{equation} \label{fml:projective-star}
\begin{split} \lin u \star \lin v \;=\; & \{ \lin w \colon w \in \lin{u,v}\withoutzero \} \\
\;=\; & P(\lin{u,v}). \end{split}
\end{equation}
Conversely, the classical coordinatisation theorem of projective geometry asserts that any projective space of rank $\geq 4$ arises, up to isomorphism, in this way from a linear space; see \cite[Chapter~VII]{Bae} or \cite[Theorem~9.2.6]{FaFr}.

The framework of projective geometry will be useful for us in the sequel but is certainly not sufficient. Indeed, the aspect of orthogonality is not taken into account. This demand has led to the notion of an {\it orthogeometry} \cite[Chapter~14]{FaFr}. Anisotropic orthogeometries correspond to Hermitian spaces in a similar way as projective spaces correspond to linear spaces \cite[Theorem~4.1.8]{FaFr}.

In the present work, we follow a simplified approach. Rather than dealing both with linear dependence and orthogonality, we restrict to the latter aspect. For $u, v \in H$, let us define $\lin u \perp \lin v$ to hold if $u \perp v$. It turns out that from the pair $(P(H), \perp)$ alone we may already reconstruct the structure of $H$.

\begin{definition}
An {\it orthoset} is a set $X$ equipped with a symmetric, irreflexive binary relation $\perp$, called the {\it orthogonality relation}.

The supremum of the cardinalities of subsets of $X$ consisting of mutually orthogonal elements is called the {\it rank} of $(X,\perp)$.
\end{definition}

\begin{example}
By the symmetry and the anisotropy of the Hermitian form, $(P(H),\perp)$ is, for any Hermitian space $H$, an orthoset. It is moreover clear from the definitions that the rank of $(P(H),\perp)$ coincides with the dimension of $H$.
\end{example}

An orthoset $(X,\perp)$ can in a natural way be understood as a closure space. For a subset $A$ of an orthoset $(X,\perp)$, we put $A\c = \{ e \in X \colon e \perp f \text{ for all } f \in A \}$. The {\it orthoclosure}, sending each $A \subseteq X$ to $A\cc = (A\c)\c$, turns out to be a closure operator on $X$ \cite{Ern} and we call a set $A$ {\it orthoclosed} if $A$ is closed with respect to $\cc$. Note that the orthoclosed subsets are exactly those of the form $B\c$, where $B$ is any subset of $X$.

\begin{definition} \label{def:linear}
An orthoset $(X,\perp)$ is called {\it linear} if, for any distinct elements $e, f \in X$, there is a $g \in X$ such that exactly one of $f$ and $g$ is orthogonal to $e$ and $\{e,f\}\c = \{e,g\}\c$.
\end{definition}

In the sequel, the following alternative formulation of linearity will be useful.

\begin{lemma} \label{lem:linearity}
An orthoset $(X,\perp)$ is linear if and only if the following conditions hold:
\begin{itemize}

\item[\rm (L1)] For any distinct elements $e$ and $f$, there is a $g \perp e$ such that $\{e,f\}\c = \{e,g\}\c$.

\item[\rm (L2)] For any distinct elements $e$ and $f$, there is, apart from $e$ and $f$, a third element in $\{e,f\}\cc$.

\end{itemize}
\end{lemma}

\begin{proof}
Assume that $(X,\perp)$ is linear. Then (L1) clearly holds. Moreover, for distinct elements $e$ and $f$ of $X$, let $g \in X$ be such that exactly one of $f$ and $g$ is orthogonal to $e$ and $\{e,f\}\c = \{e,g\}\c$. Then $g \neq f$. Moreover, $g = e$ would imply $f \perp e$, that is, $f \in \{e\}\c = \{e,f\}\c$, a contradiction. Hence $g$ is a third element of $\{e,f\}\cc$ and (L2) is shown.

Conversely, assume that (L1) and (L2) hold and let $e \neq f$. If $e \notperp f$, we infer from (L1) that the condition indicated in Definition~\ref{def:linear} is fulfilled. Assume that $e \perp f$. We first show:

($\star$) If $h \perp e$ and $h \in \{e,f\}\cc$, then $h = f$.

Indeed, assume that $h$ is as indicated and distinct from $f$. By (L1), there is a $k$ such that $k \perp f$ and $\{f,h\}\c = \{f,k\}\c$. Then $k \perp \{f,h\}\c$, hence $k \perp e$ and we conclude $k \in \{e,f\}\c$. But we also have $k \in \{f,h\}\cc \subseteq \{e,f\}\cc$, a contradiction.

By (L2), there is an element $g \in \{e,f\}\cc$ distinct from $e$ and $f$. Then $g \notperp e$, because otherwise $g = f$ by ($\star$). Consequently, by (L1), there is an $h \perp e$ such that $\{e,g\}\c = \{e,h\}\c$. Thus $h \in \{e,g\}\cc \subseteq \{e,f\}\cc$ and by ($\star$), it follows $h = f$. Hence $\{e,f\}\c = \{e,g\}\c$.
\end{proof}

The reason of the choice of the notion ``linear'' becomes apparent in the following theorem~\cite[Theorem~4.5]{PaVe1}. We call a bijection $\phi$ between orthosets $(X,\perp)$ and $(Y,\perp)$ an {\it isomorphism} if, for any $e, f \in X$, $e \perp f$ is equivalent with $\phi(e) \perp \phi(f)$.

\begin{theorem} \label{thm:linear-Hermitian}
For any Hermitian space $H$, $(P(H), \perp)$ is a linear orthoset.

Conversely, for any linear orthoset $(X,\perp)$ of rank $\geq 4$, there is a Hermitian space $H$ such that $(X,\perp)$ is isomorphic to $(P(H),\perp)$.
\end{theorem}

We note that the description of Hermitian spaces as orthosets is closely related to their lattice-theoretic characterisation. Let ${\mathcal C}(H)$ be the collection of orthoclosed subspaces of a Hermitian space $H$. Then ${\mathcal C}(H)$ is a complete, irreducible, atomistic ortholattices with the covering property \cite[Theorems~(34.2)]{MaMa}. Moreover, up to isomorphism, every such ortholattice arises in this way, provided its length is $\geq 4$ \cite[Theorems~(34.5)]{MaMa}.

Theorem~\ref{thm:linear-Hermitian} implies that the last mentioned condition, the covering property, can be weakened in such a way that only elements of height $2$ are involved~\cite[Theorem~3.10]{PaVe1}.

\begin{corollary}
Let $L$ be a complete, irreducible, atomistic ortholattice of length $\geq 4$. Assume moreover that, for any distinct atoms $p$ and $q$ there is an atom $r \perp p$ such that $p \vee q = p \vee r$. Then there is a Hermitian space $H$ such that $L$ is isomorphic to ${\mathcal C}(H)$.
\end{corollary}

We now turn to structure-preserving maps between linear orthosets and their relationship to maps between the corresponding Hermitian spaces. We are here exclusively interested in maps that preserve the orthogonality relation. For a more general approach, based on the concept of adjointability of maps, we refer the reader to the paper \cite{PaVe2}.

A map $\phi \colon X \to Y$ between orthosets is called {\it $\perp$-preserving} if, for any $e, f \in X$, $e \perp f$ implies $\phi(e) \perp \phi(f)$, and $\phi$ is called {\it $\perp$-reflecting} if, for any $e, f \in X$, $\phi(e) \perp \phi(f)$ implies $e \perp f$. We call a $\perp$-preserving and $\perp$-reflecting injection an {\it embedding}. Note that an isomorphism is the same as a surjective embedding.

Furthermore, let $V_1$ be a linear space over the sfield $F_1$ and $V_2$ a linear space over the sfield $F_2$. A map $T \colon V_1 \to V_2$ is called {\it semilinear} if $T(u+v) = T(u) + T(v)$ for any $u, v \in V_1$ and there is a homomorphism $\sigma \colon F_1 \to F_2$ such that $T(\alpha u) = \alpha^\sigma \, T(u)$ for any $u \in V_1$ and $\alpha \in F_1$; cf.~Figure~\ref{fig:semiisometry}. If in this case $\sigma$ is an isomorphism, $T$ is called {\it quasilinear}.

Let now $H_1$ be a Hermitian space over the $\star$-sfield $F_1$ and $H_2$ a Hermitian space over the $\star$-sfield $F_2$. We will mark the involutions with the respective index $1$ or $2$, and similarly for the Hermitian forms. We call $T \colon H_1 \to H_2$ a {\it semiisometry} if $T$ is semilinear with associated homomorphism $\sigma$ and there is a $\lambda \in F_2 \setminus \{0\}$ such that, for any $u, v \in H_1$,
\begin{equation} \label{fml:quasiunitary}
\herm{T(u)}{T(v)}_2 \;=\; {\herm{u}{v}_1}^\sigma \; \lambda;
\end{equation}
cf.~again Figure~\ref{fig:semiisometry}.

\begin{figure}[h!]
\vspace{2ex}
\begin{center}
\begin{tabular}{c c c }
\begin{tikzcd}
H_1 \times H_1 \arrow[r, "{+}_1"] \arrow[d, "T\times T"'] & %
H_1 \arrow[d, "T"] \\
H_2 \times H_2 \arrow[r, "{+}_2"] & H_2 \\
\end{tikzcd}& 
\begin{tikzcd}
F_1 \times H_1 \arrow[r, "\cdot_1"] \arrow[d, "{(-)}^\sigma\times T"'] & %
H_1 \arrow[d, "T"] \\
F_2 \times H_2 \arrow[r, "\cdot_2"] & H_2 \\
\end{tikzcd}& %
\begin{tikzcd}
H_1 \times H_1 \arrow[rr, "{(-,-)}_1"] \arrow[d, "T\times T"'] & &%
F_1 \arrow[r, "{(-)}^\sigma"] &F_2\arrow[dl, "{(-)\cdot \lambda}"]\\
H_2 \times H_2 \arrow[rr, "{(-,-)}_2"] & &F_2 \\
\end{tikzcd}
\end{tabular}
\end{center}
\vspace{-4ex}
\caption{For a map $T \colon H_1 \to H_2$ between Hermitian spaces to be semilinear means that the diagrams on the left and in the middle commute. Here, the operations in $H_1$ and $H_2$ are marked with subscripts $1$ and $2$, respectively. For $T$ to be a semiisometry means that in addition the diagram on the right commutes.}
\label{fig:semiisometry}
\vspace{2ex}
\end{figure}

We emphasise that the notion of a homomorphism $\sigma \colon F_1 \to F_2$ between $\star$-sfields refers to $F_1$ and $F_2$ as sfields, that is, the preservation of the involution is not required. In the present case, the involutions are by (\ref{fml:quasiunitary}) rather correlated by the equation
\begin{equation} \label{fml:quasiunitary-noch}
\alpha^{\sigma\star_2} \;=\; \lambda^{-1} \alpha^{\star_1 \sigma} \lambda
\end{equation}
for any $\alpha \in F_1$. We call $\sigma \colon F_1 \to F_2$ a {\it $\star$-homomorphism} in case when $\sigma$ is a homomorphism that does preserve the involution, which means that $\alpha^{\star_1 \sigma} = \alpha^{\sigma \star_2}$ for any $\alpha \in F_1$.

A semiisometry $T \colon H_1 \to H_2$ is called a {\it quasiisometry} if the associated homomorphism $\sigma$ is an isomorphism (of sfields), that is, if $T$ is quasilinear. A surjective quasiisometry is called {\it quasiunitary}. Note that any semiisometry is injective and hence a quasiunitary map is bijective. Finally, assume that the $\star$-sfields $F_1$ and $F_2$ coincide. Then $T$ is called {\it unitary} if $T$ is a linear isomorphism preserving the Hermitian form. That is, a unitary map is a quasiunitary map such that in (\ref{fml:quasiunitary}) we have $\sigma = \id$ and $\lambda = 1$.

We observe that semiisometries give rise to embeddings of orthosets.

\begin{lemma} \label{lem:semiisometry-induces-embedding}
Let $T \colon H_1 \to H_2$ be a semiisometry between Hermitian spaces. Then
\[ P(T) \colon P(H_1) \to P(H_2) \komma \lin u \mapsto \lin{T(u)} \]
is an embedding of orthosets.

If \/ $T$ is quasiunitary, then $P(T)$ is an isomorphism.
\end{lemma}

\begin{proof}
Note first that by semilinearity, $P(T)$ is actually definable in the indicated way. Because of (\ref{fml:quasiunitary}), $P(T)$ is $\perp$-preserving and $\perp$-reflecting. Finally, let $u, v \in H_1\withoutzero$ be such that $P(T)(\lin u) = P(T)(\lin v)$. Then for any $w \in H_1$ we have $w \perp u$ iff $T(w) \perp T(u)$ iff $T(w) \perp T(v)$ iff $w \perp v$. Thus $\lin u\c = \lin v\c$ and since one-dimensional subspaces are orthoclosed, it follows $\lin u = \lin v$ and the injectivity of $P(T)$ is shown.

If $T$ is surjective, then so is $P(T)$. Thus the additional assertion follows.
\end{proof}

\begin{example}
We may see at this point why we consider more general maps between Hermitian space than just the linear ones. Consider the case of a complex Hilbert space $H$. Any unitary map $U \colon H \to H$ induces the automorphism $P(U)$ of $P(H)$. But not every automorphism arises in this way. Let $(e_\iota)_{\iota \in I}$ be an orthonormal basis of $H$ and define $U\big(\sum_{\iota \in I} \alpha_\iota e_\iota\big) = \sum_{\iota \in I} \bar\alpha_\iota e_\iota$. Then $U$ is quasiunitary and likewise induces an automorphism of $P(H)$.
\end{example}

We next turn to a converse of Lemma~\ref{lem:semiisometry-induces-embedding}, which can be found on the basis of Piziak's work \cite{Piz}. The subsequent generalisation of Piziak's theorem \cite[Theorem~(2.2)]{Piz} was shown in \cite{PaVe2}.

\begin{theorem} \label{thm:Piziak-original}
Let $H_1$ be an at least two-dimensional Hermitian space over $F_1$ and let $H_2$ be a Hermitian space over $F_2$. Let $T \colon H_1 \to H_2$ be a non-zero semilinear map and assume that $x \perp y$ implies $T(x) \perp T(y)$ for any $x, y \in H_1$. Then $T$ is a semiisometry.
\end{theorem}

\begin{proof}
See \cite[Theorem~4.5]{PaVe2}.
\end{proof}

\begin{theorem} \label{thm:embedding-of-projective-Hermitian-spaces}
Let $H_1$ and $H_2$ be Hermitian spaces of the same dimension $n \geq 3$ and let $\phi \colon P(H_1) \to P(H_2)$ be an embedding of orthosets.
\begin{itemize}

\item[\rm (i)] Assume that $n$ is finite. Then there is a semiisometry $T \colon H_1 \to H_2$ such that $\phi = P(T)$.

\item[\rm (ii)] Assume that $\phi$ is actually an isomorphism. Then there is a quasiunitary map $T \colon H_1 \to H_2$ such that $\phi = P(T)$.

\end{itemize}
\end{theorem}

\begin{proof}
(i): Equipped with the operation $\star$ given by (\ref{fml:projective-star}), $P(H_1)$ and $P(H_2)$ are projective spaces. Note that $\lin u \star \lin v = \{\lin u, \lin v\}\cc$ because two-dimensional subspaces are splitting and hence orthoclosed.

Let $u, v, w \in H_1$ such that $\lin u \in \lin v \star \lin w$. We claim that $\phi(\lin u) \in \phi(\lin v) \star \phi(\lin w)$. Indeed, if $\lin v = \lin w$, this is trivial. Otherwise, let $w' \in H_1$ be such that $w' \perp v$ and $\lin{v,w} = \lin{v,w'}$, and let $v, w', e_1, \ldots, e_{n-2}$ be an orthogonal basis of $H_1$. Then $\phi(\lin v), \phi(\lin{w'}), \phi(\lin{e_1}), \ldots, \phi(\lin{e_{n-2}})$ is a decomposition of $H_2$ into pairwise orthogonal one-dimensional subspaces. Moreover, $u \perp e_1, \ldots, e_{n-2}$ implies $\phi(\lin u) \perp \phi(\lin{e_1}), \ldots, \phi(\lin{e_{n-2}})$ and hence $\phi(\lin u) \in \phi(\lin v) \star \phi(\lin{w'})$. Similarly, also $\phi(\lin w) \in \phi(\lin v) \star \phi(\lin{w'})$ and since, by injectivity, $\phi(\lin v) \neq \phi(\lin w)$, we conclude $\phi(\lin u) \in \phi(\lin v) \star \phi(\lin w)$ as asserted.

Moreover, the image of $\phi$ contains three mutually orthogonal elements. We conclude that $\phi$ is a non-degenerate morphism between projective spaces in the sense of Faure and Fr\" olicher \cite[6.2.1]{FaFr}.  By the Fundamental Theorem of Projective Geometry \cite[10.1.3]{FaFr}, there is a semilinear map $T$ such that $\phi = P(T)$. By Theorem~\ref{thm:Piziak-original}, $T$ is in fact a semiisometry.

(ii): For any $v, w \in H_1$, we have that $\phi(\lin v \star \lin w) = \phi(\lin v) \star \phi(\lin w)$. It follows as in part (i) that there is a semiisometry $T$ such that $\phi = P(T)$. Moreover, $\phi$ is a homomorphism in the sense of Faure and Fr\" olicher \cite[6.5.1, 6.5.2, 6.5.5]{FaFr}. By \cite[10.1.4]{FaFr}, it follows that $T$ is quasilinear, that is, a quasiisometry. From the bijectivity of $\phi$ it readily follows that also $T$ is bijective. Hence $T$ is quasiunitary.
\end{proof}

Theorem~\ref{thm:embedding-of-projective-Hermitian-spaces}(ii) indicates in which sense an inner-product space representing an orthoset according to Theorem~\ref{thm:linear-Hermitian} is essentially unique. We call Hermitian spaces $H_1$ and $H_2$ {\it isomorphic} if there is a unitary map between $H_1$ and $H_2$. Let us call $H_1$ and $H_2$ {\it equivalent} if there is a quasiunitary map between $H_1$ and $H_2$. Provided that dimensions are $\geq 3$, we have that $(P(H_1),\perp)$ and $(P(H_2),\perp)$ are isomorphic exactly if $H_1$ and $H_2$ are equivalent.

We note that ``scaling'' the Hermitian form leads to an equivalent space. Indeed, let $H$ be a Hermitian space over the $\star$-sfield $F$ and let $\lambda \in F$ be such that $\lambda \neq 0$ and $\lambda^\star = \lambda$. Define $\hermalt{\cdot}{\cdot} = \herm{\cdot}{\cdot} \lambda$. Then $\hermalt{\cdot}{\cdot}$ is a Hermitian form, the associated antiautomorphism being $^{\star'} \colon F \to F \komma \alpha \mapsto \lambda^{-1} \alpha^\star \lambda$. Let $H'$ arise from $H$ by replacing $\herm{\cdot}{\cdot}$ with $\hermalt{\cdot}{\cdot}$. Then the identity map from $H$ to $H'$ is quasiunitary.

We may use the flexibility in representing linear orthosets by Hermitian spaces to ensure additional properties; cf.~\cite[p.~207]{Hol2}. Following Gross \cite{Gro}, we call the scalar $\herm u u$ the {\it length} of a vector $u \in H$. By anisotropy, the only vector of length $0$ is the zero vector. A {\it unit vector} is a vector of length $1$.

\begin{lemma} \label{lem:scaling-the-Hermitian-form}
Let $H$ be a Hermitian space. Then we may replace the Hermitian form of $H$ with its multiple such that the resulting equivalent space $H'$ contains a unit vector.
\end{lemma}

\begin{proof}
Let $z \in H\withoutzero$ and put $\hermalt{\cdot}{\cdot} = \herm{\cdot}{\cdot} \herm{z}{z}^{-1}$. As we have argued above, the space $H'$, arising from $H$ by replacing the Hermitian form $\herm{\cdot}{\cdot}$ with $\hermalt{\cdot}{\cdot}$, is a Hermitian space equivalent to $H$. Obviously, $z$ is in $H'$ a unit vector.
\end{proof}

We briefly discuss structure-preserving self-maps. For the subsequent version of Wigner's Theorem, see also \cite{PaVe1}.

We refer to a unitary map from a Hermitian space to itself as a unitary operator and we denote the group of unitary operators on a Hermitian space $H$ by $\U(H)$. If $H$ happens to be a quadratic space, the term ``unitary'' is replaced with ``orthogonal'' and the symbol ``$\U$'' with ``$\O$''.

An automorphism of an orthoset $(X,\perp)$ is an isomorphism of $(X,\perp)$ with itself. We denote the group of automorphisms by $\Aut(X)$. The unitary operators on a Hermitian space $H$ and automorphisms of $(P(H),\perp)$ are related as follows; cf.~\cite[Section~4]{PaVe2}.

\begin{theorem} \label{thm:Wigner}
Let $H$ be a Hermitian space. For any unitary operator $U$, the map
\begin{equation} \label{fml:map-induced-by-unitary-operator}
P(U) \colon P(H) \to P(H) \komma \lin x \mapsto \lin{U(x)}.
\end{equation}
is an automorphism of $(P(H), \perp)$. The map $P \colon \U(H) \to \Aut(P(H))$ is a group homomorphism.

Conversely, let $\phi$ be an automorphism of $(P(H), \perp)$ and assume that there is an at least two-dimensional subspace $M$ of $H$ such that $\phi(\lin x) = \lin x$ for any $x \in M\withoutzero$. Then there is a unique unitary operator $U$ on $H$ such that $\phi = P(U)$ and $U|_M = \id_M$.
\end{theorem}

\begin{proof}
The first part is readily checked. To see the second part, let $\phi \colon P(H) \to P(H)$ be as indicated. If $H$ is at most two-dimensional, the assertion is trivial. Otherwise, there is by Theorem~\ref{thm:embedding-of-projective-Hermitian-spaces} a quasiunitary map $\tilde U \colon H \to H$ such that $\phi = P(U)$. Since $\tilde U$ induces on $P(M)$ the identity, there is a $\kappa \neq 0$ such that $\tilde U|_M = \kappa \, \id_M$. Let $U = \kappa^{-1} \tilde U$. Then $U$ is again a quasiunitary map such that $\phi = P(U)$. Moreover, $U|_M = \id_M$ and hence $U$ is actually unitary. The uniqueness of $U$ is clear from the fact that any other map inducing $\phi$ is a non-zero multiple of $U$.
\end{proof}

The second part of Theorem~\ref{thm:Wigner} implies that the groups consisting of the unitary operators that keep an at least two-dimensional subspace pointwise fixed can conveniently be correlated to the automorphisms of the associated orthoset.

For an orthoclosed subspace $M$ of a Hermitian space $H$, let
\[ \U(H,M) \;=\; \{ U \in \U(H) \colon U|_{M\c} = \id_{M\c} \}. \]
For an orthoclosed subset $A$ of an orthoset $X$, let
\[ \Aut(X,A) \;=\; \{ \phi \in \Aut(X) \colon \phi|_{A\c} = \id_{A\c} \}. \]

\begin{theorem} \label{thm:Wigner-restricted}
Let $H$ be a Hermitian space and let $M$ be a subspace of $H$ such that $M\c$ is at least two-dimensional. Then the map
\[ \U(H,M) \to \Aut(P(H),P(M)) \komma \; U \mapsto P(U) \]
is a group isomorphism.
\end{theorem}

\begin{proof}
Denote the indicated map by $P_M$ and note first that $P_M$ is well-defined: for any $U \in \U(H,M)$, we have that $U|_{M\c}$ is the identity, hence $P(U)|_{P(M)\c}$ is the identity as well.

By the first part of Theorem~\ref{thm:Wigner}, $P_M$ is a group homomorphism. By its second part, $P_M$ is both surjective and injective.
\end{proof}

\section{Transitive Hermitian spaces and their associated orthosets}
\label{sec:Transitivity}

We call a Hermitian space $H$ {\it transitive} if its unitary group induces a transitive action on $P(H)$. That is, for $H$ to be transitive means that, for any $u, v \in H\withoutzero$, there is a unitary operator $U$ such that $U(u) \in \lin v$. In this section, we characterise the orthosets that arise from transitive Hermitian spaces.

In the remainder of this paper, we shall tacitly make use of Lemma~\ref{lem:scaling-the-Hermitian-form}: we will assume from now on that all Hermitian spaces (including quadratic spaces) contain a unit vector.

We say that $H$ {\it admits unit vectors} if actually every one-dimensional subspace of $H$ contains a unit vector. Note that this is the case if and only if the length of every vector equals $\alpha \alpha^\star$ for some~$\alpha$.

\begin{lemma} \label{lem:transitive-1}
Let $H$ be a Hermitian space. Then the following are equivalent:
\begin{itemize}

\item[\rm (a)] $H$ is transitive.

\item[\rm (b)] Let $M_1$ and $M_2$ be subspaces of $H$ of the same finite dimension. Then there is a unitary operator $U$ such that $U(M_1) = M_2$ and $U$ is the identity on $M_1 \cap M_2$ and on $(M_1 + M_2)\c$.

\item[\rm (c)] $H$ admits unit vectors.

\end{itemize}
\end{lemma}

\begin{proof}
As we have assumed $H$ to contain a unit vector, (a) implies (c). Clearly, (b) implies (a).

To show that (c) implies (b), assume that $H$ admits unit vectors. Then any finite-dimensional subspace possesses an orthonormal basis. Extend an orthonormal basis of $M_1 \cap M_2$ to an orthonormal basis $B_1$ of $M_1$ and to an orthonormal basis $B_2$ of $M_2$. As $M_1 + M_2$ is splitting, there is a unitary operator sending $B_1$ to $B_2$ and keeping $B_1 \cap B_2$ as well as $(M_1+M_2)\c$ fixed.
\end{proof}

We will call a $\star$-sfield $F$ {\it Pythagorean} if, for any $\alpha, \beta \in F$, there is a $\gamma \in F$ such that $\alpha \alpha^\star + \beta \beta^\star = \gamma \gamma^\star$. Our terminology is consistent with the common definition of Pythagorean fields, provided that the field is understood to be equipped with the identity involution. Furthermore, we will call a $\star$-sfield $F$ {\it formally real} if, for any $\alpha_1, \ldots, \alpha_k \in F$, $k \geq 1$, we have that $\alpha_1 {\alpha_1}^\star + \ldots + \alpha_k {\alpha_k}^\star = 0$ implies $\alpha_1 = \ldots = \alpha_k = 0$. Again, for a field equipped with the identity involution, our notion coincides with the usual one. Note that a formally real $\star$-sfield has characteristic $0$.

We remark that both notions are common but in the non-commutative case definitions might differ from ours, cf.\ \cite[p.~157]{Bae}, \cite{Hol1}, as well as \cite{CrSm}.

\begin{lemma} \label{lem:transitive-2}
Let $H$ be a transitive Hermitian space of dimension $\geq 2$ over the $\star$-sfield $F$. Then $F$ is Pythagorean and formally real.
\end{lemma}

\begin{proof}
Let $\alpha, \beta \in F \setminus \{0\}$. By Lemma~\ref{lem:transitive-1}, there are orthogonal unit vectors $u, v \in H$ and for some $\gamma \in F \setminus \{0\}$, $\gamma(\alpha u + \beta v)$ is a unit vector as well. But then
\[ 1 \;=\; \herm{\gamma(\alpha u + \beta v)}{\gamma(\alpha u + \beta v)}
\;=\; \gamma \alpha \alpha^\star \gamma^\star + \gamma \beta \beta^\star \gamma^\star \]
and hence $\alpha \alpha^\star + \beta \beta^\star = \gamma^{-1} (\gamma^{-1})^\star$. We conclude that $F$ is Pythagorean.

To show that $F$ is formally real, we observe that, for any $\alpha, \beta \in F$, $\alpha \alpha^\star + \beta \beta^\star = 0$ implies $\alpha = \beta = 0$. Indeed, let again $u$ and $v$ be orthogonal unit vectors. Then $0 = \alpha \alpha^\star + \beta \beta^\star = \herm{\alpha u + \beta v}{\alpha u + \beta v}$, hence $\alpha u + \beta v = 0$ and $\alpha = \beta = 0$ as asserted. By means of the Pythagorean property, the assertion now follows by an inductive argument.
\end{proof}

Our next lemma shows that the transitive Hermitian spaces of finite dimension $\geq 2$ are, up to isomorphism, exactly the spaces $F^n$, equipped with the standard inner product, where $F$ is a Pythagorean formally real $\star$-sfield.

\begin{lemma} \label{lem:standard-inner-product}
Let $H$ be a transitive Hermitian space of finite dimension $n \geq 2$ over the $\star$-sfield $F$. Then $F$ is Pythagorean and formally real, and $H$ is isomorphic to $F^n$, equipped with the standard inner product given by
\begin{equation} \label{fml:standard-inner-product}
\Big(\begin{smm} \alpha_1 \\ \vdots \\ \alpha_n \end{smm}, \begin{smm} \beta_1 \\ \vdots \\ \beta_n \end{smm}\Big) \;=\; \alpha_1 \beta_1^\star + \ldots + \alpha_n \beta_n^\star.
\end{equation}
Conversely, for any Pythagorean formally real $\star$-sfield $F$ and any $n \geq 1$, $\,F^n$, equipped with the inner product {\rm (\ref{fml:standard-inner-product})}, is a transitive Hermitian space.
\end{lemma}

\begin{proof}
Let $H$ be as indicated. Then $F$ is Pythagorean and formally real by Lemma~\ref{lem:transitive-2}. Moreover, by Lemma~\ref{lem:transitive-1}, $H$ admits unit vectors and possesses hence an orthonormal basis. The assertions follow.

Furthermore, let $F$ be a Pythagorean formally real $\star$-sfield and $n \in \Naturals$. The property of being formally real implies that (\ref{fml:standard-inner-product}) defines a Hermitian form on $F^n$. The property of being Pythagorean implies that $H$ admits unit vectors. Clearly, $F^n$ contains a unit vector and hence, by Lemma~\ref{lem:transitive-1}, $F^n$ is transitive.
\end{proof}

Given a Pythagorean formally real $\star$-sfield $F$ and $n \geq 1$, we will assume throughout the paper that $F^n$ is the Hermitian space equipped with the Hermitian form (\ref{fml:standard-inner-product}).

We now turn to the orthosets associated with transitive Hermitian spaces. For distinct elements $e$ and $f$ of an orthoset $X$, let us call the set
\[ e \star f \;=\; \{e,f\}\cc \]
a {\it line} of $X$. Note that in the case that $X = P(H)$ for some Hermitian space $H$, this definition is consistent with (\ref{fml:projective-star}).

Furthermore, for an orthoclosed subset $A$ of an orthoset $(X,\perp)$, we will say that a subgroup $G$ of $\Aut(X)$ {\it acts transitively} on $A$ if $A$ is invariant under the action of $G$ and the action of $G|_A = \{ \phi|_A \colon \phi \in G \}$ is transitive on~$A$. Obviously, $A$ is in this case a $G$-orbit, that is, $A = G(e)$ for any $e \in A$.

\begin{definition}
We call an orthoset $(X,\perp)$
\begin{itemize}

\item[\rm (i)] {\it line-transitive} if for any two lines $\mathbcal l$ and $\mathbcal m$ there is an automorphism $\phi$ of $(X,\perp)$ such that $\phi(\mathbcal l) = \mathbcal m$;

\item[\rm (ii)] {\it line-permutable} if for any line $\mathbcal l$, $\Aut(X,{\mathbcal l})$ acts transitively on $\mathbcal l$;

\item[\rm (iii)] {\it line-symmetric} if $(X,\perp)$ is line-transitive and line-permutable;

\item[\rm (iv)] {\it flag-transitive} if for any element $e$ of a line $\mathbcal l$, and any element $f$ of a line $\mathbcal m$, there is an automorphism $\phi$ of $(X,\perp)$ such that  $\phi(e) = f$ and $\phi(\mathbcal l) = \mathbcal m$.

\end{itemize}
\end{definition}

We insert the note that any automorphism $\tau$ of an orthoset $(X,\perp)$ preserves the $\star$ operation as well as the orthocomplement. This is immediate from the fact that $\tau$ induces an automorphism of the ortholattice of orthoclosed subsets of $X$ \cite[Proposition~2.5]{Vet3}, but we include a short direct proof.

\begin{lemma} \label{lem:automorphism-preserves-structure}
Let $\tau$ be an automorphism of an orthoset $(X,\perp)$.
\begin{itemize}

\item[\rm (i)] $\tau(e \star f) = \tau(e) \star \tau(f)$ for any distinct $e, f \in X$.

\item[\rm (ii)] $\tau(A\c) = \tau(A)\c$ for any $A \subseteq X$.

\item[\rm (iii)] If $A \subseteq X$ is orthoclosed and $\tau|_{A\c} = \id_{A\c}$, then $\tau(A) = A$.

In particular, if $\tau \in \Aut(X,{\mathbcal l})$ for some line $\mathbcal l$, then $\tau(\mathbcal l) = \mathbcal l$.

\end{itemize}
\end{lemma}

\begin{proof}
(i): $\tau(e \star f) = \{ \tau(x) \colon x \perp y \text{ for all } y \perp e, f \} = \{ \tau(x) \colon \tau(x) \perp \tau(y) \text{ for all } y \text{ such} \linebreak \text{that } \tau(y) \perp \tau(e), \tau(f) \} = \{ \tilde x \colon \tilde x \perp \tilde y \text{ for all } \tilde y \perp \tau(e), \tau(f) \} = \tau(e) \star \tau(f)$.

(ii): $\tau(A\c) = \{ \tau(x) \colon x \perp A \} = \{ \tau(x) \colon \tau(x) \perp \tau(A) \} = \{ y \colon y \perp \tau(A) \} = \tau(A)\c$.

(iii): In the indicated case, we have that $\tau(A\c) = A\c$ and hence $\tau(A) = \tau(A\cc) = \tau(A\c)\c = A\cc = A$ by part (ii).
\end{proof}

\begin{lemma} \label{lem:line-symmetric-flag-transitive}
Any line-symmetric orthoset is flag-transitive.
\end{lemma}

\begin{proof}
In view of Lemma~\ref{lem:automorphism-preserves-structure}(iii), this is obvious.
\end{proof}

\begin{proposition} \label{prop:orthoset-from-transitive-Hermitian-space}
Let $H$ be a transitive Hermitian space. Then $(P(H),\perp)$ is a line-symmetric orthoset.
\end{proposition}

\begin{proof}
For linearly independent vectors $u, v \in H\withoutzero$, we have that $\lin u \star \lin v = P(\lin{u,v})$, hence any line of $P(H)$ is of the form $P(M)$ for some two-dimensional subspace $M$ of $H$.

If $H$ is at most one-dimensional, the assertion hence trivially holds. Let $H$ be at least two-dimensional. For two-dimensional subspaces $M_1$ and $M_2$ of $H$, there is by Lemma~\ref{lem:transitive-1} a unitary operator $U$ such that $U(M_1) = M_2$. This means that the automorphism $P(U)$ of $P(H)$ maps $P(M_1)$ onto $P(M_2)$. Thus $(P(H),\perp)$ is line-transitive. Finally, let $M$ once more be a two-dimensional subspace and $u, v \in M\withoutzero$. Again by Lemma~\ref{lem:transitive-1}, there is a unitary operator $U$ such that $U(\lin u) = \lin v$ and $U|_{M\c} = \id_{M\c}$. Hence  $(P(H),\perp)$ is line-permutable.
\end{proof}

Not all line-symmetric orthosets arise from inner-product spaces. For any set $X$, we call the orthoset $(X,\neq)$ {\it Boolean}.

\begin{example} \label{ex:Boolean-line-symmetric}
Any Boolean orthoset $(X,\neq)$ is line-symmetric. Indeed, $\Aut(X)$ consists of all bijections of $X$, and the lines are the two-element subsets.
\end{example}

Moreover, denote by $\emptyset$ the empty binary relation on a set $X$. The orthosets of the form $(X,\emptyset)$ are obviously exactly those of rank at most $1$.

\begin{example} \label{ex:trivial-line-symmetric}
For any set $X$, $\,(X,\emptyset)$ is line-symmetric. Indeed, $\Aut(X)$ consists again of all bijections of $X$, and if $X$ has at least two elements, the only line of $(X,\emptyset)$ is $X$.
\end{example}

We shall see that, when excluding the trivial examples, line-symmetric orthosets do arise from transitive Hermitian spaces.

\begin{lemma} \label{lem:line-symmetric-linear}
Let $(X,\perp)$ be a flag-transitive orthoset of rank $\geq 2$. Then $(X,\perp)$ is either Boolean or linear.
\end{lemma}

\begin{proof}
We first show that $(X,\perp)$ fulfils condition (L1) in Lemma~\ref{lem:linearity}. Let $e, f$ be distinct elements of $X$. As the rank of $(X,\perp)$ is assumed to be at least $2$, there are elements $c, d \in X$ such that $c \perp d$. By flag-transitivity, there is an automorphism $\tau$ such that $\tau(c \star d) = e \star f$ and $\tau(c) = e$. Then, by Lemma~\ref{lem:automorphism-preserves-structure}, $e \star f = \tau(c \star d) = \tau(c) \star \tau(d) = e \star \tau(d)$ and $\tau(d) \perp \tau(c) = e$, showing (L1).

Assume now that $(X,\perp)$ is not linear. By Lemma~\ref{lem:linearity}, condition (L2) does not hold and there are distinct elements $c, d \in X$ such that the line $c \star d$ has only two elements, that is, $c \star d = \{c,d\}$. By (L1), any line contains a pair of orthogonal elements and since orthogonal elements are distinct, we conclude that $c \perp d$. The flag-transitivity moreover implies that any other line of $(X,\perp)$ has also only two elements and hence any two distinct elements are orthogonal. This means that $(X,\perp)$ is Boolean.
\end{proof}

We arrive at the following representation of line-symmetric orthosets.

\begin{theorem} \label{thm:line-symmetric-orthoset}
Let $(X, \perp)$ be a line-symmetric orthoset of rank $\geq 4$ and assume that $(X,\perp)$ is not Boolean. Then there is a transitive Hermitian space $H$ such that $(X, \perp)$ is isomorphic to $(P(H), \perp)$.
\end{theorem}

\begin{proof}
By Lemmas~\ref{lem:line-symmetric-flag-transitive} and~\ref{lem:line-symmetric-linear}, $(X,\perp)$ is linear. By Theorem~\ref{thm:linear-Hermitian}, $(X,\perp)$ is isomorphic to $(P(H),\perp)$, where $H$ is a Hermitian space. As $(P(H),\perp)$ is line-permutable, it follows from Theorem~\ref{thm:Wigner-restricted} that $H$ is transitive.
\end{proof}

We note that, as in case of linear orthosets, we may formulate Theorem~\ref{thm:line-symmetric-orthoset} lattice-theoretically. We call an element $a$ of a lattice a {\it line} if $a$ is the join of two distinct atoms.

\begin{corollary}
Let $L$ be a complete, irreducible, atomistic ortholattice of length $\geq 4$. Assume moreover, that {\rm (i)} for any lines $a, b \in L$, there is an automorphism $\phi$ of $L$ such that $\phi(a) = b$, {\rm (ii)} for any atoms $p, q \in L$ there is an automorphism $\phi$ of $L$ such that $\phi(p) = q$ and $\phi(a) = a$ for any $a \perp p, q$. Then there is a transitive Hermitian space $H$ such that $L$ is isomorphic to ${\mathcal C}(H)$.
\end{corollary}

Summarising the present section, we have the following description of non-Boolean line-symmetric orthosets.

\begin{corollary} \label{cor:line-symmetric-orthosets}
The non-Boolean line-symmetric orthosets of rank $\geq 4$ are, up to isomorphism, exactly those of the form $(P(H),\perp)$, where $H$ is a transitive Hermitian space of dimension $\geq 4$.

In particular, the non-Boolean line-symmetric orthosets of finite rank $n \geq 4$ are, up to isomorphism, exactly the orthosets of the form $(P(F^n),\perp)$, where $F$ is a Pythagorean formally real $\star$-sfield.
\end{corollary}

\begin{proof}
The assertions hold by Theorem~\ref{thm:line-symmetric-orthoset}, Proposition~\ref{prop:orthoset-from-transitive-Hermitian-space}, and Lemma~\ref{lem:standard-inner-product}.
\end{proof}

\section{Abelian groups acting transitively on a line}
\label{sec:quadratic-orthosets}

In this section, we further narrow down the considered class of orthosets, replacing line-symmetry with the following stronger property.

\begin{definition} \label{def:quadratic-OS}
We call a line-symmetric orthoset $(X,\perp)$ {\it quadratic} if (i) $(X,\perp)$ is line-transitive and (ii), for any line $\mathbcal l$, $\Aut(X,{\mathbcal l})$ contains an abelian subgroup that acts transitively on $\mathbcal l$.
\end{definition}

It is clear that any quadratic orthoset is line-symmetric. Our guiding examples are the real Hilbert spaces.

\begin{example} \label{ex:Hilbert-space-quadratic}
For a real Hilbert space $H$, $(P(H),\perp)$ is a quadratic orthoset. Indeed, the lines of $P(H)$ correspond to the two-dimensional subspaces of $H$. Given a two-dimensional subspace $M$, the group of rotations of $M$ is abelian and acts transitively on $M$.
\end{example}

Example~\ref{ex:Hilbert-space-quadratic} can be generalised. For a finite-dimensional subspace $M$ of a quadratic space $H$, we put, as above, $\O(H,M) = \{ U \in \O(H) \colon U|_{M\c} = \id_{M\c} \}$ and we additionally define
\[ \SO(H,M) \;=\; \{ U \in \O(H,M) \colon \det U|_M = 1 \}. \]

\begin{proposition} \label{prop:quadratic-space-over-ordered-field}
Let $H$ be a transitive quadratic space. Then $(P(H),\perp)$ is a quadratic orthoset.
\end{proposition}

\begin{proof}
To begin with, we observe again that the lines of $(P(H),\perp)$ are the orthoclosed subsets of the form $P(M)$ for some two-dimensional subspace $M$ of $H$. If $H$ is at most one-dimensional, the orthoset $(P(H),\perp)$ does not contain a line and hence is quadratic. Let us assume that $H$ is at least two-dimensional.

By Proposition~\ref{prop:orthoset-from-transitive-Hermitian-space}, $(P(H),\perp)$ is line-transitive. Let us fix a two-dimensional subspace $M$. By Lemma~\ref{lem:transitive-1}, $H$ admits unit vectors and hence $M$ possesses an orthonormal basis. Let us identify the elements of $\SO(H,M)$ with their restriction to $M$ and represent them as $2 \times 2$-matrices. Then $\SO(H,M)$ consists of the matrices of the form $\begin{smm} \alpha & -\beta \\ \beta & \alpha \end{smm}$, where $\alpha^2 + \beta^2 = 1$. We observe that $\SO(H,M)$ is abelian and $P(\SO(H,M))$ acts transitively on $P(M)$. 
\end{proof}

Again, not all quadratic orthosets arise from inner-product spaces.

\begin{example} \label{ex:Boolean-quadratic}
Any Boolean orthoset is quadratic.
\end{example}

\begin{example} \label{ex:trivial-quadratic}
For any set $X$, $\,(X,\emptyset)$ is quadratic. Indeed, if $X$ has at most one element, there is no line in $(X,\emptyset)$. Otherwise, the only line of $(X,\emptyset)$ is $X$. There are abelian groups of any finite and of countably infinite cardinality and it follows from the upward L\" owenheim-Skolem Theorem that there are in fact abelian groups of any cardinality. Hence $X$ can be endowed with the structure of an abelian group that acts on itself.
\end{example}

Our aim is to show that, apart from the Boolean case and the case of low rank, the quadratic orthosets are exactly those arising from transitive quadratic spaces. The following lemma contains the key argument. Its proof is essentially given in \cite{Vet4}; here, we mention the necessary adaptations only.

\begin{lemma} \label{lem:F-is-field}
Let $H$ be an at least four-dimensional transitive Hermitian space over a $\star$-sfield $F$ and assume that $(P(H),\perp)$ is quadratic. Then $H$ is a quadratic space.
\end{lemma}

\begin{proof}
Let $M$ be a two-dimensional subspace of $H$. By Theorem~\ref{thm:Wigner-restricted}, the map $\U(H,M) \to \Aut(P(H),P(M)) \komma U \mapsto P(U)$ is a group isomorphism. Furthermore, by assumption, there is an abelian subgroup of $\Aut(P(H),P(M))$ that acts transitively on $P(M)$. We conclude that there is an abelian subgroup $\N$ of $\U(H,M)$ such that $P(\N)$ acts transitively on $P(M)$.

By Lemma~\ref{lem:transitive-2}, $F$ is formally real and hence has characteristic $0$. In order to show that $^\star$ is the identity, we proceed as in case of \cite[Lemma~18]{Vet4}. We conclude that $F$ is commutative, that is, $H$ is quadratic.
\end{proof}

Thus we have the following representation of quadratic orthosets.

\begin{theorem} \label{thm:quadratic-orthoset}
Let $(X, \perp)$ be a quadratic orthoset of rank $\geq 4$ and assume that $(X,\perp)$ is not Boolean. Then there is a transitive quadratic space $H$ such that $(X, \perp)$ is isomorphic to $(P(H), \perp)$.
\end{theorem}

\begin{proof}
This is the consequence of Theorem~\ref{thm:line-symmetric-orthoset} and Lemma~\ref{lem:F-is-field}.
\end{proof}

We add again a lattice-theoretic formulation of Theorem~\ref{thm:quadratic-orthoset}. We denote the automorphism group of an ortholattice $L$ by $\Aut(L)$. Moreover, for an element $a$ of a lattice $L$, we write $\downset a = \{ b \in L \colon b \leq a \}$.

\begin{corollary}
Let $L$ be a complete, irreducible, atomistic ortholattice of length $\geq 4$. Assume moreover that {\rm (i)} for any lines $a, b \in L$, there is an automorphism $\phi$ of $L$ such that $\phi(a) = b$, {\rm (ii)} for any line $a \in L$, $\{ g \in \Aut(L) \colon g|_{\downset a\c} = \id_{\downset a\c} \}$ contains an abelian subgroup acting transitively on $\downset a$. Then there is a transitive quadratic space $H$ such that $L$ is isomorphic to ${\mathcal C}(H)$.
\end{corollary}

We may summarise our results as follows.

\begin{corollary} \label{cor:quadratic-orthosets}
The non-Boolean quadratic orthosets of rank $\geq 4$ are, up to isomorphism, exactly those of the form $(P(H),\perp)$, where $H$ is a transitive quadratic space of dimension $\geq 4$.

In particular, the non-Boolean quadratic orthosets of finite rank $n \geq 4$ are, up to isomorphism, exactly the orthosets of the form $(P(F^n),\perp)$, where $F$ is a Pythagorean formally real field equipped with the identity involution.
\end{corollary}

\begin{proof}
The assertions hold by Theorem~\ref{thm:quadratic-orthoset}, Proposition~\ref{prop:quadratic-space-over-ordered-field}, and Lemma~\ref{lem:standard-inner-product}.
\end{proof}

We conclude the section with the observation that, for any quadratic orthoset $(X,\perp)$ of rank $\geq 4$, the abelian subgroup of $\Aut(X,\mathbcal l)$ acting transitively on a line $\mathbcal l$ is unique.

\begin{proposition} \label{prop:quadratic-space-over-ordered-field-D}
Let $(X,\perp)$ be a quadratic orthoset of rank $\geq 4$. Then there is, for any line $\mathbcal l$, a unique abelian subgroup of $\Aut(X,\mathbcal l)$ acting transitively on $\mathbcal l$. Moreover, any two such groups are conjugate subgroups of $\Aut(X)$.
\end{proposition}

\begin{proof}
If $(X,\perp)$ is Boolean, the statement is clear. Assume that $(X,\perp)$ is non-Boolean and let $H$ be the quadratic space representing $(X,\perp)$ according to Theorem~\ref{thm:quadratic-orthoset}.

Let $M$ be a two-dimensional subspace of $H$. We identify the elements of $\O(H,M)$ again with the matrix representation of their restriction to~$M$ with respect to some fixed orthonormal basis of $M$. Then $\SO(H,M) = \big\{ \begin{smm} \alpha & -\beta \\ \beta & \alpha \end{smm} \colon \alpha^2 + \beta^2 = 1 \big\}$. We have that $P(\SO(H,M))$ is abelian and acts transitively on $P(M)$.

We have to show that $P(\SO(H,M))$ is the only subgroup of $\Aut(P(H),P(M))$ with these properties. Let $\N$ be an abelian subgroup of $\O(H,M)$ such that $P(\N)$ acts transitively on $P(M)$. In view of the isomorphism between the groups $\O(H,M)$ and $\Aut(P(H),P(M))$ according to Theorem~\ref{thm:Wigner-restricted}, we have to show that $\N = \SO(H,M)$.

We observe next that the elements of $\O(H,M)$ are of the form
\[ {\rm (a)} \begin{smm} \alpha & -\beta \\ \beta & \alpha \end{smm}
\quad\text{or}\quad
{\rm (b)} \begin{smm} \alpha & \beta \\ \beta & -\alpha \end{smm},
\quad \text{where $\alpha^2 + \beta^2 = 1$.} \]
We have that $\begin{smm} 0 & 1 \\ 1 & 0 \end{smm}$ does not commute with any matrix of the form (a) or (b) if $\alpha, \beta \neq 0$. Hence $\begin{smm} 0 & 1 \\ 1 & 0 \end{smm} \in \N$ would imply that $\N$ consists of matrices all of whose entries are in $\{-1, 0, 1\}$, contradicting the transitivity assumption. We conclude that $\begin{smm} 0 & 1 \\ 1 & 0 \end{smm} \notin \N$ and similarly we see that $\begin{smm} 0 & -1 \\ -1 & 0 \end{smm} \notin \N$. But $\N$ must contain an element sending $\linbig{\begin{smm} 1 \\ 0 \end{smm}}$ to $\linbig{\begin{smm} 0 \\ 1 \end{smm}}$, that is, either $\begin{smm} 0 & -1 \\ 1 & 0 \end{smm}$ or $\begin{smm} 0 & 1 \\ -1 & 0 \end{smm}$. As the square is in both cases $\begin{smm} -1 & 0 \\ 0 & -1 \end{smm}$, we have that $\begin{smm} 0 & -1 \\ 1 & 0 \end{smm}, \, \begin{smm} 0 & 1 \\ -1 & 0 \end{smm} \in \N$.

Since $\begin{smm} 0 & -1 \\ 1 & 0 \end{smm}$ does not commute with any matrix of the form (b), $\N$ consists solely of matrices of the form (a). This means $\N \subseteq \SO(H,M)$.

Conversely, let $\alpha^2 + \beta^2 = 1$. Let $U \in \N$ be a map sending $\linbig{\begin{smm} 1 \\ 0 \end{smm}}$ to $\linbig{\begin{smm} \alpha \\ \beta \end{smm}}$. Then either $U = \begin{smm} \alpha & -\beta \\ \beta & \alpha \end{smm}$ or $U = \begin{smm} -\alpha & \beta \\ -\beta & -\alpha \end{smm}$. As $\begin{smm} -1 & 0 \\ 0 & -1 \end{smm} \in \N$, we conclude that actually both these elements are in $\N$. Hence $\SO(H,M) \subseteq \N$.

For any two-dimensional subspaces $M_1$ and $M_2$ of $H$, $\O(H,M_1)$ and $\O(H,M_2)$ are conjugate subgroups of $\O(H)$. The last assertion follows.
\end{proof}

\section{Minimal orthosets}
\label{sec:Minimal-orthosets}

We have considered three different classes of orthosets: linear, line-symmetric, and quadratic ones. In this final section, we raise the question whether there is, among the orthosets of a given type and rank, one that is minimal in the sense that it embeds into all the other ones. It turns out that there is, for each finite $n \geq 4$, a quadratic orthoset of rank $n$ that embeds into any other quadratic orthoset of the same rank. For the linear or line-symmetric orthosets, a similar statement does not hold.

It is clear that the issue is closely related to the existence of homomorphisms between the corresponding sfields.

\begin{proposition} \label{prop:homomorphism-between-scalar-sfields}
Let $H_1$ be a Hermitian space over the $\star$-sfield $F_1$ and let $H_2$ be a Hermitian space over the $\star$-sfield $F_2$. Assume that $H_1$ and $H_2$ have the same finite dimension $n \geq 3$ and let $\phi \colon P(H_1) \to P(H_2)$ be an embedding of orthosets. Then there is a homomorphism $\sigma \colon F_1 \to F_2$.

If in this case $F_2$ is a field equipped with the identity involution, then so is $F_1$.
\end{proposition}

\begin{proof}
By Theorem~\ref{thm:embedding-of-projective-Hermitian-spaces}, there is a semiisometry $T \colon H_1 \to H_2$ such that $\phi = P(T)$. In particular, there is a homomorphism $\sigma \colon F_1 \to F_2$ associated with $T$.

Assume that $F_2$ is commutative and the involution $^{\star_2}$ on $F_2$ is the identity. Clearly, also $F_1$ is then commutative. Moreover, by (\ref{fml:quasiunitary-noch}), $\alpha^\sigma = \alpha^{\sigma\star_2} = \lambda^{-1} \alpha^{\star_1 \sigma} \lambda =  \alpha^{\star_1 \sigma}$ for any $\alpha \in F_1$. It follows that the involution $^{\star_1}$ on $F_1$ is the identity as well.
\end{proof}

As regards the converse direction, we observe the following.

\begin{lemma} \label{lem:embedding-induced-by-sfield-homomorphism}
Let $\sigma \colon F_1 \to F_2$ be a $\star$-homomorphism between Pythagorean formally real \linebreak $\star$-sfields, and let $n \geq 1$. Then $\sigma$ induces an embedding of orthosets $\phi \colon P({F_1}^n) \to P({F_2}^n)$.
\end{lemma}

\begin{proof}
We may clearly define $\phi \colon P({F_1}^n) \to P({F_2}^n) \komma \linBig{\begin{smm} \alpha_1 \\ \vdots \\ \alpha_n \end{smm}} \mapsto \linBig{\begin{smm} {\alpha_1}^{\sigma} \\ \vdots \\ {\alpha_n}^{\sigma} \end{smm}}$. Obviously, $\phi$ is injective and $\perp$-preserving as well as $\perp$-reflecting.
\end{proof}

Let us first deal with the (non-Boolean) quadratic orthosets. The key question is whether there is a Pythagorean formally real field that embeds into any other such field.

In what follows, the field of real numbers $\Reals$, as well as subfields of $\Reals$, are understood to be equipped with the identity involution.

\begin{proposition} \label{prop:Hilbert-field}
There is a least Pythagorean subfield of \/ $\Reals$.
\end{proposition}

\begin{proof}
We readily check that the intersection of Pythagorean subfields of \/ $\Reals$ is again a Pythagorean field.
\end{proof}

The least Pythagorean subfield of the reals is called the {\it Hilbert field}. We will see that it has the property to embed into any other Pythagorean formally real field. We start with a basic lemma on the extension of field homomorphisms.

\begin{lemma} \label{lem:embedding-extension}
Let $G \subseteq G'$ be an extension of fields and $\psi \colon G \to F$ be a homomorphism of fields. Let $f(x)\in G[x]$ be an irreducible polynomial over $G$, let $\alpha\in G' \setminus G$ be a root of $f$, and let $\beta\in F$ be a root of $\psi(f)(x) \in F[x]$. Then there is a homomorphism $\tilde{\psi} \colon G[\alpha] \to F$ such that $\psi \subseteq \tilde{\psi}$ and $\tilde{\psi}(\alpha) = \beta$.
\end{lemma}

\begin{proof}
We have that $G[\alpha] = \{g(\alpha)\in G' \colon g \text{ is a polynomial over } G\}$. We claim that we may define
\[ \tilde{\psi} \colon G[\alpha] \to F \komma g(\alpha) \mapsto \psi(g)(\beta). \]
Indeed, let $g_1(x), g_2(x)\in G[x]$ be such that $g_1(\alpha) = g_2(\alpha)$. Since $f(x)$ is irreducible over $G$ and $f(\alpha) = 0$, it follows that $g_1 - g_2 = f \cdot h$ for some $h(x) \in G[x]$. Hence $\psi(g_1)(\beta) - \psi(g_2)(\beta) = \psi(f)(\beta) \, \psi(h)(\beta) = 0$ as $\beta$ is a root of $\psi(f)$. Therefore $\psi(g_1)(\beta) = \psi(g_2)(\beta)$ and we conclude that the map $\tilde{\psi}$ is well defined.

Clearly, $\tilde{\psi}$ is a homomorphism with the required properties.
\end{proof}

\begin{theorem} \label{thm:Hilbert-field}
Let $R$ be the Hilbert field and let $F$ be any Pythagorean formally real field. Then there is a homomorphism $\nu \colon R \to F$.
\end{theorem}

\begin{proof}
Let $\mathcal{C}$ be the collection of all embeddings of a subfield of $R$ into $F$. Then $\mathcal C$ is non-empty because there is an embedding of $\Rationals$ into $F$. Moreover, $\mathcal C$ is partially ordered by extension and every chain $\{ \varphi_i \colon i \in I \}$ in $\mathcal{C}$ has the upper bound $\bigcup_{i \in I} \varphi_i$. We conclude from Zorn's lemma that there is a maximal element $\psi \colon G \to F$ in $\mathcal{C}$. We claim that $G = R$.
 
Assume to the contrary that $G$ is a proper subfield of $R$. Then $G$ is not Pythagorean, hence there are $a,b\in G$ such that  the polynomial $f(x) = x^2-(a^2+b^2)\in G[x]$ has no root in $G$ and $f$ is therefore irreducible over $G$.  As $R$ is Pythagorean, the polynomial $f$ has a root $c\in R \setminus G$. As also $F$ is Pythagorean, there is a root $e\in F$ of the polynomial $\psi(f)(x) = x^2 - (\psi(a)^2 + \psi(b)^2) \in F[x]$. By Lemma~\ref{lem:embedding-extension}, there is a homomorphism $\tilde{\psi} \colon G[c]\to F$ extending $\psi$, in contradiction to the maximality of $\psi$.
\end{proof}

Let us apply this result to specify quadratic orthosets of a ``minimal'' type.

\begin{theorem} \label{thm:minimal-quadratic-orthoset}
Let $R$ be the Hilbert field and let $n \geq 4$ be finite. Then $(P(R^n),\perp)$ is a quadratic orthoset of rank $n$ that embeds into any non-Boolean quadratic orthoset of rank $n$.
\end{theorem}

\begin{proof}
By Lemma~\ref{lem:standard-inner-product} and Proposition~\ref{prop:quadratic-space-over-ordered-field}, $(P(R^n),\perp)$ is a non-Boolean quadratic orthoset. By Corollary~\ref{cor:quadratic-orthosets}, any further non-Boolean quadratic orthoset of rank $n$ is isomorphic to $(P(F^n),\perp)$ for some Pythagorean formally real field $F$. By Theorem~\ref{thm:Hilbert-field}, $R$ embeds into $F$ and hence, by Lemma~\ref{lem:embedding-induced-by-sfield-homomorphism}, $(P(R^n),\perp)$ embeds into $(P(F^n),\perp)$.
\end{proof}

In case of the other two classes of orthosets that we have considered in this paper, we must give a negative answer to the question whether we may identify minimal members. We restrict to the case of a finite rank $\geq 4$.

Let us consider the (non-Boolean) line-symmetric orthosets.

\begin{example} \label{ex:no-minimal-line-symmetric-orthsets}
Let $\Quaternions$ be the $\star$-sfield of quaternions, equipped with the standard conjugation given by $(a + bi + cj + dk)^\star = a - bi - cj - dk$. Consider the subfield $F = \Rationals + \Rationals i + \Rationals j + \Rationals k$ of $\Quaternions$. Equipped with $^\star$, $F$ is a $\star$-sfield as well.

For any $\alpha \in F$, $\alpha \alpha^\star$ is the sum of squares of four rational numbers. Indeed, for $\alpha = a + bi + cj + dk$ we have $\alpha \alpha^\star = a^2 + b^2 + c^2 + d^2$. It follows that $F$ is formally real. Moreover, for any $\alpha, \beta \in F$, $\alpha \alpha^\star + \beta \beta^\star$ is the sum of squares of eight rational numbers and hence of the form $\frac{r_1^2 + \ldots + r_8^2}{t^2}$ for some $r_1, \ldots, r_8, t \in \Naturals$. By Lagrange's four square theorem, there are $s_1, s_2, s_3, s_4 \in \Naturals$ such that $\alpha \alpha^\star + \beta \beta^\star = \frac{s_1^2 + s_2^2 + s_3^2 + s_4^2}{t^2}$. Putting $\gamma = \frac 1 t (s_1 + i s_2 + j s_3 + k s_3)$, we thus get $\alpha \alpha^\star + \beta \beta^\star = \gamma \gamma^\star$. Hence $F$ is Pythagorean.

Finally, the Hilbert field $R$ does not embed into $F$. Indeed, otherwise there would be some $\alpha = a + bi + cj + dk \in F$ such that $\alpha^2 = 2$. But this means $a^2 - b^2 - c^2 - d^2 + 2a(bi + cj + dk) = 2$. Then $a = 0$ implies $-b^2-c^2-d^2 = 2$, and $a \neq 0$ implies $b = c = d = 0$ and hence $a^2 = 2$, both a contradiction.
\end{example}

For $n \geq 4$, assume that there is a non-Boolean line-symmetric orthoset $(\check X,\perp)$ with the property that is embeds into any other non-Boolean line-symmetric orthoset of rank $n$. In accordance with Corollary~\ref{cor:line-symmetric-orthosets}, let $\check F$ be the Pythagorean formally real $\star$-sfield such that $(\check X,\perp)$ is isomorphic to $(P({\check F}^n),\perp)$. Then $(P({\check F}^n),\perp)$ embeds into $(P(R^n),\perp)$, where $R$ is the Hilbert field. By Proposition~\ref{prop:homomorphism-between-scalar-sfields}, $\check F$ is a field equipped with the identity involution and there is a homomorphism $\sigma \colon \check F \to R$. In particular, $\check F$ is a Pythagorean formally real field as well. If $\sigma$ was not surjective, $R$ would contain a proper subfield that is Pythagorean and formally real, contradicting Proposition~\ref{prop:Hilbert-field}. Hence $\check F$ is isomorphic to $R$. Again by Corollary~\ref{cor:line-symmetric-orthosets} and Proposition~\ref{prop:homomorphism-between-scalar-sfields}, this means that $R$ embeds into any other Pythagorean formally real $\star$-sfield. However, by Example~\ref{ex:no-minimal-line-symmetric-orthsets}, this is not the case. We conclude that there is no line-symmetric orthoset with the indicated minimality property.

Finally, let us consider the whole class of linear orthosets.

Recall that the {\it u-invariant} $u(F)$ of a field $F$ is the largest dimension that a quadratic space over $F$ may have \cite[Chapter XI, §6]{Lam}.

\begin{example} \label{ex:no-minimal-linear-orthsets}
For a prime number $p > 2$, let $F$ be an algebraically closed field of characteristic $p$. For any $n \geq 1$, let $F_{(n)}$ be the $n$-fold iterated Laurent field, that is, $F_{(n)} = F((t_1))((t_2)) \ldots ((t_n))$. Clearly, $F_{(n)}$ has likewise characteristic $p$. Moreover, $u(F_{(n)}) = 2^n$ \cite[Example 6.2(7)]{Lam} and this means that there is a $2^n$-dimensional quadratic space over $F_{(n)}$.
\end{example}

For $n \geq 4$, assume that there is a linear orthoset $(\check X,\perp)$ of rank $n$ with the property that is embeds into any other linear orthoset of rank $n$. In accordance with Theorem~\ref{thm:linear-Hermitian}, let $\check F$ be the $\star$-sfield such that $(\check X,\perp)$ is isomorphic to $(P({\check H}),\perp)$, where $\check H$ is an $n$-dimensional Hermitian space over $\check F$. Then $(\check X,\perp)$ embeds into $(P(H),\perp)$ whenever $H$ is a Hermitian space over some $\star$-sfield $F$. By Proposition~\ref{prop:homomorphism-between-scalar-sfields}, $\check F$ embeds in this case into $F$. However, Example~\ref{ex:no-minimal-linear-orthsets} shows that there are quadratic spaces over fields of any characteristic $\neq 2$ and of arbitrary dimension. That is, $F$ may have an arbitrary characteristic $\neq 2$. We conclude that there is no linear orthoset of rank $n$ that is, in the sense indicated above, minimal.

\subsection*{Acknowledgements} J.P.\ and T.V.'s research was funded in part by the Austrian Science Fund (FWF) 10.55776/ PIN5424624 and the Czech Science Foundation (GACR) 25-20013L. M.K.\ acknowledges the support by the Austrian Science Fund (FWF): project I~4579-N and the Czech Science Foundation (GA\v CR): project 20-09869L.

\end{document}